\documentclass[12pt,reqno]{amsart}
\usepackage{amsmath, amssymb, amsfonts,upgreek}
\usepackage[all]{xy}
\numberwithin{equation}{section}
\newtheorem{thm}{Theorem}[section]
\newtheorem{prop}[thm]{Proposition}
\newtheorem{lem}[thm]{Lemma}

\newtheorem{remark}[thm]{Remark}
\newtheorem{cor}[thm]{Corollary}
\usepackage[a4paper, top = 1.2in, bottom = 1.2in, left = 1.2in, right = 1.2in]{geometry}

\numberwithin{equation}{section}

\newcommand{\N}{\mathbb{N}}
\newcommand{\Q}{\mathbb{Q}}

\newcommand{\Z}{\mathbb{Z}}
\newcommand{\mbH}{\mathbb{H}}

\newcommand{\ord}{\textrm{ord}}
\newcommand{\Hom}{\mathrm{Hom}}

\newcommand{\mfa}{\mathfrak{a}}

\newcommand{\mfm}{\mathfrak{m}}
\newcommand{\mfn}{\mathfrak{n}}

\newcommand{\mfW}{\mathbf{W}}

\newcommand{\mH}{\mathrm{H}}
\newcommand{\SL}{\mathrm{SL}}
\newcommand{\GL}{\mathrm{GL}}

\newcommand{\ab}{\textrm{ab}}
\newcommand{\lra}{\longrightarrow}
\newcommand{\ra}{\rightarrow}
\newcommand{\ras}{\twoheadrightarrow}

\newcommand{\psmat}[4]{\bigl( \begin{smallmatrix} #1 & #2 \\ #3 & #4 \end{smallmatrix} \bigr)}

\begin{document}
\title{A $2$-adic control theorem for modular curves}
\author{Narasimha Kumar}
\email{narasimha.kumar@iith.ac.in}

\address{Department of Mathematics \\
Indian Institute of Technology Hyderabad \\
Ordnance Factory Estate \\
Yeddumailaram 502205 \\
Telangana, INDIA.}

\begin{abstract}
We study the behaviour of ordinary parts of the homology modules of modular curves,
associated to a decreasing sequence of congruence subgroups $\Gamma_1(N2^r)$ for $r \geq 2$, and prove
a control theorem for these homology modules.
\end{abstract}
\maketitle

\section{Introduction}
Hida theory studies the modular curves associated to the following congruence subgroups, for primes 
$p \geq 5$ and $(p , N) = 1$,
\begin{equation*}
 \dotsm \subset \Gamma_1(Np^r) \subset \dotsm \subset \Gamma_1(Np).\tag{$\ast$}
\end{equation*}
Let $Y_r$ denote the Riemann surface associated to the congruence subgroup $\Gamma_1(Np^r)$. One of 
the important results in Hida theory~\cite{Hid86} is that the projective limit of ordinary parts of the 
homology modules, i.e., $W^{\rm ord}:= \underset{r}{\varprojlim}\ \mH_1(Y_r, \Z_p)^{\rm ord}$, is a 
free $\Lambda$-module of finite rank and 
\begin{equation*}
 W^{\rm ord} / \mfa_r W^{\rm ord} =  \mH_1(Y_r, \Z_p)^{\rm ord}, 
\tag{$\ast\ast$} 
\end{equation*}
$\textrm{for all}\ r \geq 1$, where $\mfa_r$ denotes the augmentation ideal of $\Z[[1+p^r\Z_p]]$ and $\Lambda=\Z_p[[1+p\Z_p]]$.
In~\cite{Eme99}, Emerton gave a proof of these results above for primes $p \geq 5$, using algebraic topology of the 
Riemann surfaces  $Y_r$. 

Emerton's proof for $p \geq 5$ holds for $p = 3$ with $N > 1$ verbatim, 
but for $p = 2$ we show that similar results hold  only after
passing to smaller congruence subgroups. Moreover, there is no restriction on $N$, i.e., $N$ can be equal to $1$ (unlike when $p=3$)
(cf. Theorem~\ref{mainthmone} in the text).
As a consequence of these results, we proved control theorems for ordinary $2$-adic families of modular forms, see~\cite{GK12}.
Some amount of calculations will be omitted and the reader should refer to those in~\cite{Eme99} for more details.


\section{Preliminaries}
Throughout this note, let $p = 2$, $q = 4$, and $N \in \N$ such that $(p,N)=1$. We look at the modular curves associated 
to the following congruence subgroups
\begin{equation*}
  \dotsm \subset \Gamma_1(Np^r) \subset \dotsm \subset \Gamma_1(Nq).
\end{equation*}
If we take the homology with $\Z$-coefficients of the tower of modular curves, we get a tower 
of finitely generated free abelian groups
\begin{equation}
\label{chain}
 \dotsm \ra \Gamma_1(Np^r)^{\ab} \ra \dotsm \ra \Gamma_1(Nq)^{\ab},
\end{equation}
because for $r \geq 2$, $\mH_1(\Gamma_1(Np^r) \backslash \mbH, \Z)= \Gamma_1(Np^r)^{\ab}$, where $\mbH$ denotes the upper half-plane. To understand \eqref{chain}, we introduce the congruence subgroups 
for $r \geq 2$:
\begin{equation*}
 \Phi_r^2 = \Gamma_1(Nq)\cap\Gamma_0(p^r). 
\end{equation*}
Clearly, we have
$ \Gamma_1(Np^r) \subset \Phi_r^2 \subset \Gamma_1(Nq)$
and $\Gamma_1(Np^r)$ is a normal subgroup of $\Phi_r^2$.
For $r \geq 2$, we define $\Gamma_r := \text{Ker}\ (\Z_p^{\times} \ras (\Z_p/p^r\Z_p)^{\times})$, which is a subgroup 
of $\Gamma_2$ with index ${p}^{r-2}$. Set $\Gamma: =\Gamma_2$.
 
We define a morphism of groups
\begin{equation*}
 \Phi_r^2 \stackrel{\eta_r}{\lra} \Gamma / \Gamma_r
\end{equation*}
via the formula
\begin{equation}
 \label{dsurj}
  \quad \psmat{a}{b}{c}{d} \lra d \bmod {\Gamma_r}. 
\end{equation}

\begin{lem}
\label{etasur}
The map $\eta_r$ is surjective.
\end{lem}

\begin{proof}
Given a $\bar{d} \in \Gamma/\Gamma_r$, we can take a lift $d$ of $\bar{d}$ of the form $1+kqN$ for some $k \in \Z$,  because for any $\alpha, \beta \in \Gamma, \alpha \equiv \beta\  \pmod {\Gamma_r}$ if and only if 
$\alpha-\beta \in p^r \Z_p $. Now, take $c$ to  be $Np^r$. Clearly $(c,d)=1$, and hence 
there exists $a,b \in \Z$ such that $ad-bc=1$. We see that $\alpha = \psmat{a}{b}{c}{d} \in \Phi_r^2$ and
$\eta_r(\alpha) = \bar{d}$.
\end{proof}

\begin{remark}
\label{isosubgroup}
The restriction of $\eta_r$ to $\Phi_r^2 \cap \Gamma^0(p)$, which we denote by  Res($\eta_r$),
is also surjective onto $\Gamma / \Gamma_r$. Moreover, we have the following commutative diagram
\begin{equation*}
\xymatrix@1{
\Phi_{r+1}^2 \ar@{->>}[r]^{\eta_{r+1} \quad} \ar[d]^{t^{-1}-t}_{\wr} & \Gamma/\Gamma_{r+1} \ar@{->>}[d]\\
\Phi_r^2\cap \Gamma^0(p) \ar@{->>}[r]^{\quad \text{Res}(\eta_r)} & \Gamma/\Gamma_r,
}
\end{equation*}  
where the group $\Gamma^0(p) = \{ \psmat{a}{b}{c}{d} \in \SL_2 (\Z) \ \lvert\  \> b \equiv 0 \pmod 2 \}$ 
and $t = \psmat{1}{0}{0}{p}$.
\end{remark}

By Lemma~\ref{etasur}, we have the following short exact sequence of groups
\begin{equation*}
1 \ra \Gamma_1(Np^r) \ra \Phi_r^2 \stackrel{\eta_r}{\ra} \Gamma /\Gamma_r \ra 1.
\end{equation*}
The action of $\Phi_r^2$ on $\Gamma_1(Np^r)$ by conjugation induces an action of  
$\Phi_r^2/\Gamma_1(Np^r) = $ $ \Gamma/\Gamma_r$ on $\Gamma_1(Np^r)^{\ab}$. 
Thus $\Gamma$ acts naturally  on $\Gamma_1(Np^r)^{\ab}$.
The morphisms in the chain
\begin{equation*}
\dotsm \ra \Gamma_1(Np^r)^{\ab} \ra \dotsm \ra \Gamma_1(Nq)^{\ab}
\end{equation*}
are clearly $\Gamma$-equivariant.

If $r \geq s > 1$, we denote by $\Phi_r^s$ the subgroup of $\Phi_r^2$ containing $\Gamma_1(Np^r)$ 
whose quotient by $\Gamma_1(Np^r)$ equals $\Gamma_s/\Gamma_r$, i.e., $\Phi_r^s := \Gamma_1(Np^s) \cap \Gamma_0(p^r)$.
Moreover, we have
\begin{equation*}
  \Gamma_1(Np^r)^{\ab} \ra \Phi_r^{s \> \ab} \ra \Gamma_s/\Gamma_r \ra 1.
\end{equation*}

For any $s >1$, let $\gamma_s$ denote a topological generator of $\Gamma_s$. Then the augmentation ideal $\mfa_s$ of $\Lambda = \Z_p[[\Gamma]]$
is a principal ideal generated by $\gamma_s-1$. Similarly, for $i > 0$, $\Gamma_{s+i} =\overline{\langle \gamma_s^{p^i} \rangle}$ and 
$\mfa_{s+i} = (\gamma_s^{p^i}-1)$. Clearly, for any $r \geq s > 1$, the augmentation ideal of $\Z[\Gamma_s/\Gamma_r]$ is $\mfa_s$,
and 
\begin{equation*}
  \mfa_s\Gamma_1(Np^r)^{\ab} = [\Phi_r^s,\Gamma_1(Np^r)]/[\Gamma_1(Np^r),\Gamma_1(Np^r)] 
                               \subset \Gamma_1(Np^r)^{\ab},
\end{equation*}
and the last inclusion follows since $\Gamma_1(Np^r)$ is a normal subgroup of $\Phi_r^s$.
The extension
\begin{equation*}
  1 \ra \Gamma_1(Np^r)/[\Phi_r^s,\Gamma_1(Np^r)] \ra \Phi_r^s/[\Phi_r^s,\Gamma_1(Np^r)] 
                                                 \ra \Gamma_s/\Gamma_r \ra 1
\end{equation*}
is a central extension of a cyclic group, thus the middle group is abelian, implying that
\begin{equation*}
[\Phi_r^s,\Phi_r^s]=[\Phi_r^s,\Gamma_1(Np^r)].
\end{equation*}
The equality holds because of $\Phi_r^s \supseteq \Gamma_1(Np^r)$ and the fact 
that the commutator subgroup of the group $\Phi_r^s/[\Phi_r^s,\Gamma_1(Np^r)]$ is trivial.

\begin{remark}
\label{isosubgroup-1}
The following diagram is commutative
\begin{equation*}
\xymatrix{
  \frac{\Phi_r^s\cap \Gamma^0(p)}{\Gamma_1(Np^r)\cap \Gamma^0(p)} \ar[r]^{\quad i} 
  \ar[dr]^{\sim} & \frac{\Phi_r^s}{\Gamma_1(Np^r)} \ar[d]^{\sim} \\
  & \frac{\Gamma_s}{\Gamma_r}.
}
\end{equation*}
The diagonal map is an isomorphism, by Remark~\ref{isosubgroup}.
Since $\Gamma_s/\Gamma_r$ is finite, we see that the inclusion $i$ is an isomorphism.
This remark is useful in proving Lemma~\ref{VUcomm}.
\end{remark}

To prove Theorems~\ref{Hida} and~\ref{mainthmone}, we need to understand the images of 
these morphisms
\begin{equation*}
  \Gamma_1(Np^r)^{\ab} \ra \Gamma_1(Np^s)^{\ab}
\end{equation*}
in the chain of homology groups as in~\eqref{chain}. Unfortunately, we do not have a good characterization these images for $r \geq s > 1$ in general, and so we cannot get a good description of the projective limit. 
This morphism can be factored as 
\begin{equation*}
  \Gamma_1(Np^r)^{\ab} \ras \Gamma_1(Np^r)^{\ab}/\mfa_{s} 
                            \hookrightarrow \Phi^{s \> \ab}_r 
\lra \Gamma_1(Np^s)^{\ab},
\end{equation*}
and the problem is that the second and third morphisms  may not be isomorphisms, in general.

Hida observed that if one applies a certain projection operator arising from the Atkin $U$-operator 
to all these modules then they become isomorphisms, in which case we have a good control over the images
of the morphisms in~\eqref{chain}. So we now define the Atkin $U$-operator and study their properties.

\section{Hecke operators}
Suppose $G,H$ are two subgroups of a group $T$, and $t \in T$ such that $[G:t^{-1}Ht \cap G]<\infty$. 
Then one has 
\begin{equation*}
  G^{\ab}\overset{V}{\lra} (t^{-1}Ht\cap G)^{\ab} \overset{\sim}{\lra} (H \cap tGt^{-1})^{\ab}  \lra H^{\ab},
\end{equation*}
where $V$ is the transfer map, the isomorphism is given by conjugating with $t$,
and the last morphism is induced by $H \cap tGt^{-1} \hookrightarrow H$.
Taking the composition of all these we obtain a morphism
\begin{equation*}
  [t]: G^{\ab} \ra H^{\ab},
\end{equation*}
the ``Hecke operator'' corresponding to $t$.

In our case, take $T = \GL_2(\Q)$, $G = H =$ a congruence subgroup of $\SL_2$ of level divisible by $p$, and
$t = \psmat{1}{0}{0}{p}$. We denote the corresponding Hecke operator by $U_2$.
For $A = \psmat{a}{b}{c}{d} \in \Phi_r^s$, we see that
\begin{equation*}
t^{-1} A t = \psmat{a}{bp}{c/p}{d} \quad \text{and} \quad 
t A t^{-1} = \psmat{a}{b/p}{cp}{d}. 
\end{equation*}

\begin{remark}
 \label{dconju}
Observe that $(1,1)$, $(2,2)$-entries of $A$ and of $ t ^{\pm 1} A t^{\mp 1}$ are the 
same.
\end{remark}

It is easy to see that 
$t^{-1} \Phi_r^s t \cap \Phi_r^s = \Phi_r^s \cap \Gamma^0(p), \Phi_r^s 
       \cap t \Phi_r^s t^{-1} = \Phi ^s _{r+1}, $
where the group $\Gamma^0(p)$ is as in Remark~\ref{isosubgroup}.
Thus, the Atkin $U$-operator (resp. $U'$-operator) is by definition the composition
\begin{equation}
\label{local-equation}
\xymatrix{
\Phi_r^{s \> \ab} \ar^{V \qquad}[r] & (\Phi_r^s \cap \Gamma^0(p))^{\ab} 
                  \ar[r]_{\ \ \quad \backsim}^{\quad\quad\  t-t^{-1}} & 
                  \Phi_{r+1}^{s \> \ab} \ar[r] & \Phi_r^{s \> \ab},
}
\end{equation}
(resp., the composition of the first two of above morphisms). 

\begin{lem}
\label{Ucomm}
Suppose that $r \geq s > 1, r' \geq s' > 1, r \geq r', s \geq s'$, so that $\Phi^s_r 
              \subset \Phi^{s'}_{r'}$. 
Then the following diagram commutes
\begin{equation*}
\xymatrix{
\Phi_r^{s \> \emph{ab}} \ar[r] \ar[d]^{U'} & \Phi_{r'}^{s' \> \emph{ab}} \ar[d]^{U'} \\
\Phi_{r+1}^{s \> \emph{ab}}  \ar[r] & \Phi_{r'+1}^{s' \> \emph{ab}}.
}
\end{equation*}
Thus, the Atkin $U$-operator commutes with the morphism $\Phi_r^{s \> \emph{ab}} \ra \Phi_{r'}^{s' \> \emph{ab}}$.
\end{lem}
\begin{proof}
The proof is similar to the proof of~\cite[Lem. 3.1]{Eme99}. 
The final statement follows from~\eqref{local-equation}, since 
the Atkin $U$-operator, by definition, is the composition of $U'$-operator and 
the morphism induced by the inclusion of groups $\Phi_{r+1}^s \subset \Phi_r^s$.

\end{proof}

\begin{cor}
\label{coker}
For $r \geq s >1$, each $\Phi^{s \> \ab}_r$ is a $\Z[U]$-module via the action of $U$
and morphisms between these modules (arising from the inclusions)  are morphisms of $\Z[U]$-modules. 
Hence, the cokernels of these morphisms acquire a $\Z[U]$-module structure.
\end{cor}

Suppose $\pi$ denote the morphism
$ \pi : \Phi_r^{s \> \ab} \lra \Phi_{r-1}^{s \> \ab}$
and $\pi'$ for the morphism
$ \pi' : \Phi_{r+1}^{s \> \ab} \lra \Phi_{r}^{s \> \ab}$.
Then, by Lemma~\ref{Ucomm}, we have
\begin{equation}
 \label{UPiComm-1}
   U'\circ \pi =\pi' \circ U' = U \in \text{End}_{\Z}(\Phi^{s \> \ab}_r).
\end{equation}
By the definition of $U'$, we see that $\pi \circ U' = U \in \text{End}_{\Z}(\Phi^{s \> \ab}_{r-1})$.

 By Corollary~\ref{coker}, the cokernel of the morphism 
$\Gamma_1(Np^r)^{\ab}$  $\ra \Phi_{r}^{s \> \ab}$, for $r \geq s > 1$, is a $\Z[U]$-module
and this cokernel is isomorphic to the group $\Gamma_s/\Gamma_r$. Hence, the group 
$\Gamma_s/\Gamma_r$ is a $\Z[U]$-module. Observe that $\Phi_r^r = \Gamma_1(Np^r)$.

\begin{lem}
\label{lemmultiq}
The operator $U$ acts on $\Gamma_s/\Gamma_r$ as multiplication by $p$.
\end{lem}
\begin{proof}
The operator $U$ acts on $\Gamma_s/\Gamma_r$ as a multiplication by $p$ if and only if it acts on 
$\frac{\Phi_{r}^{s \> \ab}}{\Gamma_1(Np^r)^{\ab}}$ as $\bar{A} \mapsto \bar{A}^p$. The operator $U$ is the composition 
of the following morphisms:
\begin{equation}
\label{multiq}
 \begin{matrix}
  \quad & \frac{\Phi_{r}^{s \> \ab}}{\Gamma_1(Np^r)^{\ab}} \stackrel{V}{\lra} & \frac{(\Phi_r^s \cap
  \Gamma^0(p))^{\ab}}{(\Gamma_1(Np^r)\cap \Gamma^0(p))^{\ab}} \stackrel{\quad t-t^{-1}}{\lra} &  
   \frac{\Phi_{r+1}^{s\ \ab}}{\Phi_{r+1}^{r\ \ab}} \quad \lra & \frac{\Phi_{r}^{s \> \ab}}{\Gamma_1(Np^r)^{\ab}}\\
 &&&&\\
   \quad & \quad \bar{A} \qquad \longmapsto & \qquad \quad \bar{A}^p \quad \qquad \longmapsto & \quad t\bar{A}^pt^{-1} \quad 
   \longmapsto & t\bar{A}^pt^{-1}.
\end{matrix}
\end{equation}
Let $\{ \alpha_i = \psmat{1}{i}{0}{1} \}_{i=0}^{p-1}$ 
be the coset representatives of the group $\Phi_r^s \cap \Gamma^0(p)$ in $\Phi_r^s$. If we use these 
representatives to define the map in~\eqref{multiq}, then the transfer map looks like 
$\bar{A} \mapsto \bar{A}^p$. By Remark~\ref{dconju}, $tA^pt^{-1}$ and $A^p$ represent the 
same coset mod $\Gamma_1(Np^r)^{\ab}$ and hence we are done.
\end{proof}

We would like to define an action of $\Gamma$ on $\Phi^{s \> \ab}_r$ and call it the nebentypus action. 
This can be done as follows: For $r \geq 2$, if $\bar{d} \in \Gamma/\Gamma_r$, then choose 
an element $\alpha = \psmat{a}{b}{c}{d}$ of $\SL_2(\Z)$ such that $p^{r+1} \mid c$ and $p \mid b$, i.e., 
$\alpha \in \Phi^2_{r+1} \cap \Gamma^{0}(p)$. Such an $\alpha$ exists, because
\begin{equation*}
  \Phi_{r+1}^2\cap \Gamma^0(p)\twoheadrightarrow \Gamma / \Gamma_{r+1} \twoheadrightarrow \Gamma / \Gamma_{r}.
\end{equation*}
The nebentypus action of $d$  on $\Phi_r^{s \> \ab}$ 
is given by conjugation by $\alpha$. This action is well-defined because if $\alpha_1$ and $\alpha_2$ denote 
two lifts of $\bar{d}$, then $\alpha_1^{-1}\alpha_2 \in \Gamma_1(Np^{r+1})\cap\Gamma^0(p) \subseteq \Phi_r^s$
and hence for any element $x \in \Phi_r^s$, $\alpha_1^{-1} \alpha_2 x \alpha_2^{-1} \alpha_1 =
x $ in $\Phi_r^{s\> \ab}$. Now we shall show that the actions of $U$ and $\Gamma$ commutes. 

\begin{lem}
\label{UGaComm}
If $r \geq s > 1$, the actions of $U$ and $\Gamma$ commutes on $\Phi_r^{s \> \emph{\ab}}$.
\end{lem}
\begin{proof}
Though the proof of this lemma is similar to the proof of~\cite[Lem. 3.5.]{Eme99}, here we make some remarks
in between, hence we briefly recall its proof. It is easy to see that $ \alpha (\Phi_r^s\cap \Gamma^0(p)) \alpha^{-1}= \Phi_r^s\cap \Gamma ^0(p)$ for any 
$\alpha \in \Phi_{r+1}^{1} \cap \Gamma^0(p)$, since $\alpha \Phi_r^s \alpha^{-1} \subseteq \Phi_r^s$. Look 
at the following commutative diagram
\begin{equation*}
\xymatrix{
\Phi_r^{s \> \ab} \ar[r]^{\alpha - \alpha^{-1}\qquad} \ar[d]^{V} & 
   \Phi_r^{s \> \ab} \ar[d]^{V} \\
(\Phi_r^s\cap \Gamma ^0(p))^{\ab} \ar[r]^{\alpha - \alpha^{-1}\qquad \qquad} \ar[d]^{{t-t^{-1}}} & 
   (\alpha (\Phi_r^s\cap \Gamma ^0(p)) \alpha^{-1})^{\ab}=(\Phi_r^s\cap \Gamma ^0(p))^{\ab}\ar[d]^{\alpha t 
    \alpha^{-1}(-)\alpha t^{-1} \alpha^{-1}}\\
\Phi_{r+1}^{s \> \ab} \ar[r]^{\alpha - \alpha ^{-1} \qquad} \ar[d] & 
   (\alpha \Phi_{r+1}^s \alpha^{-1})^{\ab}=\Phi_{r+1}^{s \> \ab}\ar[d] \\
\Phi_r^{s \> \ab} \ar[r]^{ \alpha - \alpha^{-1} \qquad} & 
   (\alpha \Phi_r^s \alpha ^{-1})^{\ab} = \Phi_r^{s \> \ab}.
}
\end{equation*}
The top square in the diagram above commutes because if $\{ \gamma_1,\ldots,\gamma_q \}$ form  a set 
coset representatives for the group $\Phi_r^s \cap \Gamma^0(p)$ in $\Phi_r^s$, so is the set  
$\{ \alpha \gamma_1 \alpha^{-1}, \ldots, \alpha \gamma_q \alpha^{-1} \}$.
Observe that, this diagram commutes even 
if $\alpha \in \Phi_{r}^{1} \cap \Gamma^0(p)$. The last square commutes by the functoriality of the transfer map. 

We now prove the commutativity of the middle square, i.e., the map 
\begin{equation}
\label{localequ}
  {\alpha t \alpha^{-1}(-)\alpha t^{-1} \alpha^{-1}}: (\Phi_r^s\cap \Gamma ^0(p))^{\ab} \ra \Phi_{r+1}^{s \> \ab} 
\end{equation}
is $t - t^{-1}$. If $g \in \Phi_r^{s} \cap \Gamma^0(p)$, then
\begin{equation*}
  \alpha t \alpha^{-1} g \alpha t^{-1}\alpha^{-1} = (\alpha t \alpha^{-1} t^{-1}) 
                       t g t^{-1} (\alpha t\alpha^{-1}t^{-1})^{-1}.
\end{equation*}
Since $\alpha t \alpha^{-1} t^{-1} \in \Gamma_1(Np^{r+1})$ for $\alpha \in \Phi_{r+1}^{1} \cap \Gamma^0(p)$, we see
that the conjugation by $\alpha t \alpha^{-1} t^{-1}$ induces identity on $\Phi_{r+1}^{s \> \ab}$\ (because elements 
of $\Phi_{r+1}^s$ do commute in $\Phi_{r+1}^{s \> \ab}$). 

In the above diagram composition of the vertical morphisms on either side are the operator $U$ and it 
commutes with the automorphism of $\Phi^s_r$ induced by conjugation by $\alpha$, but we know  
$\Gamma$ acts on $\Phi_r^s$ by conjugation by such elements $\alpha$.
\end{proof}
Observe that the inclusion $\Gamma_1(Np^r) \subseteq \Phi^s_r$
gives rise to the another transfer map
\begin{equation*}
\Phi^{s \> \ab}_r \stackrel{V}{\lra} \Gamma_1(Np^r)^{\ab} 
\end{equation*}

\begin{lem}
\label{VUcomm}
The transfer morphism $V: \Phi_r^{s \> \emph{ab}} \rightarrow \Gamma_1(Np^r)^{\emph{ab}}$ commutes 
with the action of $U$ on its source and target.
\end{lem}

\begin{proof}
It suffices to prove that the following diagram (in which $V$ denotes the transfer maps between 
various abelianizations) commutes:
\begin{equation*}
\xymatrix{
\Phi_r^{s \> \ab} \ar[r]^{V} \ar[d]^{V} & \Phi_r^{r \> \ab} \ar[d]^V \\
(\Phi_r^s\cap \Gamma ^0(p))^{\ab} \ar[r]^{V} \ar[d]^{t-t^{-1}} & 
(\Phi_r^r\cap \Gamma ^0(p))^{\> \ab}\ar[d]^{t-t^{-1}} \\
\Phi_{r+1}^{s \> \ab} \ar[r]^V & \Phi_{r+1}^{r\> \ab}.
}
\end{equation*}
The top square in the diagram above commutes because of functoriality of the transfer map. The commutativity 
of the bottom square follows by the following calculation.

If $\sigma_d = \psmat{a}{b}{c}{d}$, where $d$ runs through coset representatives of $\Gamma_r$ in $\Gamma_s$,
forms a set of coset representatives for the group $\Gamma_1(Np^r)\cap \Gamma^0(p)$ in $\Phi^s_r \cap \Gamma^0(p)$,
then so are $t \sigma_d t^{-1} = \psmat{a}{b/p}{cp}{d}$ 
for the group $\Gamma_1(Np^r)$ in $\Phi_r^s$ (by Remark~\ref{isosubgroup-1}).
\end{proof}

In this section, we have defined the $U$-operators for the congruence subgroups $\{ \Phi_r^s \}$ and 
proved that morphisms between these congruence subgroups respects the action of $U$ and 
this action commutes with the action of $\Gamma$.  

\section{Ordinary parts}
Let $A$ be a commutative finite $\Z_p$-algebra and $U$ be a non-zero element of $A$.  
It well-known that $A$ factors as a product of local rings.
Let $A^{\ord}$ denote the product of all those local rings of $A$ 
in which the projection of $U$ is a unit. This is a flat $A$-algebra. 

Let $M$ be any module in the abelian category of $\Z_p[X]$-modules which are finitely generated as $\Z_p$-modules.
In this case, we take $A$ to be  the image of $\Z_p[X]$ in $\text{End}_{\Z_p}(M)$,
which is a finite $\Z_p$-algebra,
and $U$ to be the image of $X$. 
We define
\begin{equation*}
 M^{\ord}:= M \otimes_A A^{\ord} 
\end{equation*}
and call this the ordinary part of $M$. 
Observe that taking ordinary parts is an exact functor on our abelian category.

If we consider $X$ to be the $U$-operator corresponding to the prime $p$, we may consider the ordinary part 
of the $\Z_p$-homology of the curve $Y_r$, i.e., the module $(\Gamma_1(Np^r)^{\ab} \otimes \Z_p)^{\ord}$, 
which is a $\Gamma$-module by Lemma~\ref{UGaComm}. 

We have the following theorem for the prime $p = 2$, which is similar to Theorem $3.1$ in~\cite{Hid86} 
for $p \geq 5$ and for the congruence subgroups $\Gamma_1(Np^r)$ for $r \geq 1$. 
\begin{thm}
  \label{Hida}
  If $r \geq s > 1$, then the morphism of abelian groups
\begin{equation*}
  (\Gamma_1(Np^r) \otimes \Z_p)^{\emph{\ord}} / \mfa_s \ra (\Gamma_1(Np^s) \otimes \Z_p)^{\emph{\ord}}
\end{equation*}
is an isomorphism.
\end{thm}
\begin{proof}
We shall show that
\begin{equation}
\label{local}
(\Gamma_1(Np^r)^{\ab} \otimes \Z_p)^{\ord} / \mfa_s \overset{\sim}{\lra} (\Phi_r^{s \> \ab} \otimes \Z_p )^{\ord} 
\overset{\sim}{\lra} (\Gamma_1(Np^s)^{\ab} \otimes \Z_p)^{\ord}.
\end{equation}
 If $\pi : \Phi^{s \> \ab}_{r} \ra \Phi^{s \> \ab}_{r-1} $
is the morphism induced by the inclusion $\Phi^{s} _{r} \subset \Phi^{s}_{r-1}$, 
then
\begin{equation*}
U'\circ \pi = U \in \mathrm{End}(\Phi^{s \> \ab}_{r}),\>\pi \circ U'= U 
\in \mathrm{End}(\Phi^{s \> \ab}_{r-1}).
\end{equation*}
By Lemma~\ref{Ucomm}, we have the following diagram
\begin{equation*}
\xymatrix{
\Phi^{s \> \ab}_{r-1} \ar[r]^{\pi} \ar[d]^{U'} \ar[dr]^{U} & \Phi^{s \> \ab}_{r-2} \ar[d]^{U'} \\
\Phi^{s \> \ab}_{r} \ar[r]^{\pi} \ar[d]^{U'} \ar[dr]^{U} & \Phi^{s \> \ab}_{r-1} \ar[d]^{U'} \\
\Phi^{s \> \ab}_{r+1} \ar[r]^{\pi}  & \Phi^{s \> \ab}_{r}.
} 
\end{equation*}
The existence of $U'$ implies that upon tensoring over $\Z_p$ and taking the ordinary parts $\pi$ 
induces an isomorphism (and $U^{-1} \circ U'$ provides an inverse to $\pi$)
\begin{equation*}
(\Phi^{s \> \ab}_{r}\otimes \Z_p)^{\text{\ord}} =  (\Phi^{s \> \ab}_{r-1}\otimes \Z_p)^{\text{\ord}}.
\end{equation*}
By induction on $r$, we obtain the second isomorphism in~\eqref{local}, i.e.,
\begin{equation*}
(\Phi^{s \> \ab}_{r}\otimes \Z_p)^{\ord} = (\Phi^{s \> \ab}_{s}\otimes \Z_p)^{\text{\ord}} 
          = (\Gamma_1(Np^s)^{\ab}\otimes \Z_p)^{\ord}.
\end{equation*}

To prove the first isomorphism consider the short exact sequence
\begin{equation*}
 1 \ra \Gamma_1(Np^r)^{\ab} /\mfa_s \ra \Phi^{s \> \ab}_{r} \ra (\Gamma_s/\Gamma_r) \ra 1.
\end{equation*}
By tensoring this sequence with $\Z_p$ and then taking the ordinary parts to obtain 
\begin{equation*}
1 \ra (\Gamma_1(Np^r)^{\ab}\otimes \Z_p)^{\ord} /\mfa_s \ra (\Phi^{s \> \ab}_{r} \otimes \Z_p)^{\ord} \ra (\Gamma_s/\Gamma_r)^{\ord} \ra 1 ,
\end{equation*}
because $\Z_p$ is flat as a $\Z$-module and ordinary parts preserves exactness.
By Lemma~\ref{lemmultiq}, the operator $U$ acts on $\Gamma_s/\Gamma_r$ as multiplication by $p$ 
and so is a nilpotent operator, as $\Gamma_s/\Gamma_r$ is a $p$-torsion group. Thus 
$(\Gamma_s/\Gamma_r)^{\ord} = 0$, and hence the Theorem follows.
\end{proof}

\section{Iwasawa modules}
We have the following inverse system indexed by natural numbers $r \geq 2 $,
\begin{equation*}
\dotsm \ra \Gamma_1(Np^r)^{\ab}\otimes \Z_p \ra \dotsm \ra \Gamma_1(Np^2)^{\ab}
           \otimes \Z_p.
\end{equation*}
Define the Iwasawa module by
\begin{equation*}
\mfW := \varprojlim_{r \geq 2} \ \Gamma_1(Np^r)^{\ab}\otimes \Z_p.
\end{equation*}
The profinite group $\Gamma$ acts on the $\Z_p$-module $\Gamma_1(Np^r)^{\ab}\otimes \Z_p$ through 
its finite quotient $\Gamma/\Gamma_r$. Thus the Iwasawa module $\mfW$ becomes a module over 
the completed group algebra 
\begin{equation*}
\Lambda := \Z_p[[\Gamma]] = \varprojlim_{r \geq 2} \Z_p[\Gamma/\Gamma_r].
\end{equation*}

Though the Iwasawa module $\mfW$ is difficult to understand,
by Theorem \ref{Hida}, we can understand the ordinary part of $\mfW$ very well. 
To make the statement clear, let us slightly abstract the situation. 

Let $\{ M_r \}_{r \geq 2}$ be a system of $\Lambda$-modules. Further, assume that each $M_r$ is point-wise 
fixed by $\Gamma_r$ and hence a module  over $\Lambda/\mfa_r \Lambda= \Z_p [\Gamma/\Gamma_r]$. 
For each $r \geq s \geq 2$, we have a map 
$M_r \ra M_s$ such that it factors via
\begin{equation*}
  M_r/\mfa_sM_r \ra M_s.
\end{equation*}
Define $W := \varprojlim_{r \geq 2} M_r$. We have 
a collection of maps $W \ra M_r$ for each $r \geq 2$ and they factor as
\begin{equation*}
  W/\mfa_rW \ra M_r.
\end{equation*}

\begin{prop}
\label{Keylemma}
  Assume that each $M_r$ is p-adically complete and for each $r \geq s \geq 2$, 
  $M_r/\mfa_sM_r \ra M_s$ is an isomorphism. Then $W/\mfa_sW \ras M_s $ is an isomorphism.
\end{prop}

\begin{proof}
For $r \geq s \geq 2$, the maps $M_r \ra M_s$  are surjective, and hence the canonical map 
from $W \ra M_s $ is also surjective. We shall show that the kernel is $\mfa_sW$. 

Since each $M_r$ is $p$-adically complete and is point-wise fixed by $\Gamma_r$, we have 
$M_r = \varprojlim_i M_r/\mfn^iM_r$, where $\Gamma_r= \overline{\langle \gamma_r \rangle}$ and $\mfn = (\gamma_r-1,p)$, i.e.,
each $M_r$ is $\mfn$-adically complete.

By induction on $i$, we get that $\gamma_s^{p^i}-1/\gamma_s-1 \in (\gamma_s-1,p)^{i}$. In particular,  
we have $\gamma_2^{p^{r-2}}-1/\gamma_2-1 \in \mfm=(\gamma_2-1,p)$. Hence,  $\mfm^{p^{r-2}} \subseteq ((\gamma_2-1)^{p^{r-2}},p) \subseteq \mfn \subseteq \mfm = (\mfa_2,p)$. As a result, we see that
each $M_r$ is $\mfm$-adically complete, since they are  $\mfn$-adically complete. Once we have that
each $M_r$ is $\mfm$-adically complete, then proving the injectivity of the above map is quite 
similar to the proof of ~\cite[Prop. 5.1]{Eme99}.
\end{proof}
The following Theorem is an immediate consequence of the Proposition above.

\begin{thm}
 \label{mainthmone}
   For any $r \geq 2$, we have
   \begin{equation*}
     \mfW^{\emph{\ord}}/\mfa_{r}\mfW^{\emph{\ord}} \cong (\Gamma_1(Np^r)^{\emph{\ab}}\otimes \Z_p)^{\emph{\ord}}  
   \end{equation*}
   is the $\Gamma_r$-co-invariants of $\mfW^{\emph{\ord}}$.
\end{thm}

\begin{proof}
  This follows from Proposition~\ref{Keylemma} together with Theorem~\ref{Hida}
\end{proof}

The above Theorem is a key ingredient for the proof of the Theorem~\ref{mainfree}.
The $\Lambda$-module $\mfW^{\ord}$ is a compact $\Lambda$-module (under the projective limit of the $p$-adic topologies 
on each module $\Gamma_1(Np^r)^{\ab} \otimes \Z_p$, which are free of finite rank over $\Z_p$,
and also since $\mfW^{\ord}$ is a direct factor of $\mfW$).

Furthermore, Theorem~\ref{mainthmone} implies that the projective limit topology on $\mfW^{\ord}$ coincides with 
its $\mfm$-adic topology (where $\mfm=(\mfa_2,p) \subset \Lambda$ denotes the maximal ideal of 
$\Lambda$), because the kernels of the projections $\Lambda \ra \Z_p/p^r\Z_p[\Gamma/\Gamma_r]$
are co-final with the sequence of ideals $\mfm^{r}$ in $\Lambda$. Thus $\mfW^{\ord}$ is a $\Lambda$-module, 
compact in its $\mathfrak{m}$-adic topology such that
\begin{equation*}
  \mfW^{\ord}/\mathfrak{m} = \mfW^{\ord}/(\mfa_2,p) = (\Gamma_1(Nq)^{\ab} \otimes \Z_p/p)^{\ord}
\end{equation*}
is a finite dimensional $\Z_p/p\Z_p$-module, of dimension $d$ (say). By Nakayama's lemma, 
we have that $\mfW^{\ord}$ is a finitely generated $\Lambda$-module with a minimal generating 
set has cardinality $d$.  We have the following theorem for the prime $p = 2$, which is similar to the main theorem in~\cite{Hid86} for $p \geq 5$.
\begin{thm}[Main Result]
 \label{mainfree}
   The module $\mfW^{\emph{\ord}}$ is free of finite rank over $\Lambda$, and its $\Lambda$-rank is equal to $d$.
\end{thm}
As a corollary, we see that, for $r \geq 2$, the $\Z_p$-rank of the free $\Z_p$-module $(\Gamma_1(Np^r)^{\ab} \otimes \Z_p)^{\ord}$ is $d$. In particular, these $\Z_p$-ranks are independent of $p^r$ in the level. Using this result,
we have proved control theorems for ordinary $2$-adic families of modular forms, see~\cite{GK12}.
The classical versions of this theorem for $p = 2, 3$ do not seem to be explicitly available in 
the literature, though an ad{\`e}lic version of it can be found in~\cite{Hid88}.


 
\section{Reflexivity results}
To prove Theorem~\ref{mainfree}, it enough to show that $\mfW^{{\ord}}$ is a reflexive $\Lambda$-module~\cite{NSW08}. We show this by considering the duality theory of the modules $(\Gamma_1(Np^r)^{\ab} \otimes \Z_p)^{\ord}$ and showing that they are reflexive as $\Z_p[\Gamma/\Gamma_r]$-modules.
Now, we briefly recall the notion of reflexivity,  and the necessary results. For more details, see~\cite[\S 6]{Eme99}. 

Suppose that $R$ is a commutative ring, $G$ is a finite group, and $M$ is a left $R[G]$-module. 
Let $N$ be any $R$-module. Then $\Hom_R(M,N)$ is a right $R[G]$-module, via 
\begin{equation*}
(f*g)(x) := f(g^{-1}  x).
\end{equation*}
Since the ring $R[G]$  is naturally a bi-module over itself, via the ring multiplication,
$R[G] \otimes_R N$ is an $R[G]$-bi-module, making $\Hom_{R[G]}(M, R[G] \otimes_R N )$ a 
right $R[G]$-module.
\begin{lem}[\cite{Eme99}]
There is a canonical isomorphism of right $R[G]$-modules
\begin{equation*}
\Hom_R(M,N) = \Hom_{R[G]}(M, R[G]\otimes_R N).
\end{equation*}
\end{lem}


In particular, when $N=R$, we see that $M^*$ and $\Hom_{R[G]}(M,R[G])$ are canonically 
isomorphic as right $R[G]$-modules, where $M^{*}:=\Hom_{R}(M,R)$, the $R$-dual of $M$. 
The analogue of the above  lemma for right $R[G]$-modules is also true. Hence,
\begin{equation*}
  \Hom_R(M^*,R) = \Hom_{R[G]}(M^*, R[G]).
\end{equation*}
are canonically isomorphic as left $R[G]$-modules.

By definition of $M^{*}$, there is a natural morphism of $R$-modules $ M \ra \Hom_R(M^*,R)$, which is also
a morphism of left $R[G]$-modules. If this natural morphism of $R$-modules is an isomorphism, then 
we say that $M$ is a reflexive $R$-module.  Thus we have proved:
\begin{lem}
 \label{Equality2}
 If $M$ is a left $R[G]$-module which is reflexive as an $R$-module, then $M$ is reflexive as an $R[G]$-module.
\end{lem}
The crux of this Lemma is that to check the reflexivity of $R[G]$-module $M$ over  $R[G]$, 
it is enough to check it over $R$. Now we need to understand how to use the reflexivity 
results for modules over $\Z_p[\Gamma/\Gamma_r]$ to show the reflexivity of $\mfW^{\ord}$ as a 
$\Lambda$-module. 

\section{Proof of Theorem~\ref{mainfree}}
For $r \geq 2$, and $N \in \N$ such that $(p,N)=1$. We define the cohomology of $Y_r$ as
\begin{equation*}
\mH^1(Y_r,\Z_p) := \Hom_{\Z}(\Gamma_1(Np^r)^{\ab}, \Z_p)  
                       = \Hom_{\Z_p}(\Gamma_1(Np^r)^{\ab}\otimes \Z_p, \Z_p).
\end{equation*}
The ring $\Lambda$ acts on $\Gamma_1(Np^r)^{\ab}\otimes \Z_p$ through its quotient
$\Lambda_r := \Lambda /\mfa_r =\Z_p[\Gamma/\Gamma_r].$
More generally, if $r \geq s > 1$ then the ring $\Lambda_s$ is equal to $\Lambda_r/\mfa_s$, hence
$\Lambda_r \ras \Lambda_s$. Thus we get the following sequence of morphisms of $\Lambda_r$-modules
\begin{equation*}
\begin{split}
  \Hom_{\Lambda_r}(\Gamma_1(Np^r)^{\ab}\otimes \Z_p,\Lambda_r)  \ra & \> 
   \Hom_{\Lambda_r}(\Gamma_1(Np^r)^{\ab}\otimes \Z_p,\Lambda_r)/ \mfa_s  \\
   \ra \Hom_{\Lambda_r}(\Gamma_1(Np^r)^{\ab}\otimes \Z_p,\Lambda_s) 
       & =\Hom_{\Lambda_s}(\Gamma_1(Np^r)^{\ab}\otimes \Z_p/\mfa_s,\Lambda_s).
\end{split}
\end{equation*}
If $M$ is any $\Z_p[U]$-module, which is finitely generated as a $\Z_p$-module, then 
so is the $\Z_p$-dual $M^{*}:= \Hom_{\Z_p}(M,\Z_p)$. Here $M^*$ is a $\Z_p[U]$-module via the dual action of $U$. 
Clearly $(M^*)^{\ord} = (M^{\ord})^{*}$, 
i.e., taking ordinary parts commutes with duals. Thus we may take ordinary parts 
of the above diagram of homomorphisms to obtain a diagram
\begin{equation*}
 \begin{split}
  \Hom_{\Lambda_r}((\Gamma_1(Np^r)^{\ab}\otimes \Z_p)^{\ord},\Lambda_r) & \lra
  \Hom_{\Lambda_r}((\Gamma_1(Np^r)^{\ab}\otimes \Z_p)^{\ord},\Lambda_r)/\mfa_s  \\
  & \lra \Hom_{\Lambda_s}((\Gamma_1(Np^r)^{\ab}\otimes \Z_p)^{\ord}/ \mfa_s,\Lambda_s). 
\end{split}
\end{equation*}
By Theorem~\ref{Hida}, we have
\begin{equation*}
\begin{split}
  \Hom_{\Lambda_r}((\Gamma_1(Np^r)^{\ab}\otimes \Z_p)^{\ord},\Lambda_r) & \lra
  \Hom_{\Lambda_r}((\Gamma_1(Np^r)^{\ab}\otimes \Z_p)^{\ord},\Lambda_r)/\mfa_s  \\
  & \lra \Hom_{\Lambda_s}((\Gamma_1(Np^s)^{\ab}\otimes \Z_p)^{\ord},\Lambda_s). 
\end{split}
\end{equation*}

\begin{lem}
\label{MainLemma}
 The morphism
\begin{equation*}
\begin{split}
\Hom _{\Lambda_r}((\Gamma_1(Np^r)^{\emph{\ab}}\otimes \Z_p)^{\emph{\ord}},\Lambda_r)/\mfa_s & \ra \Hom _{\Lambda_s}((\Gamma_1(Np^s)^{\emph{\ab}}\otimes \Z_p)^{\emph{\ord}},\Lambda_s)
\end{split}
\end{equation*}
is an isomorphism.
\end{lem}
\begin{proof}
By Lemma~\ref{VUcomm}, we may restrict $V$ to the ordinary parts to obtain a morphism
\begin{equation*}
(\Phi_r^{s \> \ab} \otimes \Z_p )^{\ord} \stackrel{V}{\lra} (\Gamma_1(Np^r)^{\ab} \otimes \Z_p)^{\ord}.
\end{equation*}
Look at the following commutative diagram
\begin{equation*}
\xymatrix{
 \Hom_{\Z_p}( (\Gamma_1(Np^r)^{\ab} \otimes \Z_p)^{\ord},\Z_p) \ar[r]^{\sim} \ar[d]^{V^*} & \Hom_{\Lambda_r}( 
 (\Gamma_1(Np^r)^{\ab} \otimes \Z_p)^{\ord},\Lambda_r) \ar[d] \\
 \Hom_{\Z_p}( (\Phi^{s \> \ab}_r \otimes \Z_p)^{\ord},\Z_p) \ar[dd]^{\wr} & \Hom_{\Lambda_r}( 
 (\Gamma_1(Np^r)^{\ab} \otimes \Z_p)^{\ord},\Lambda_r)/\mfa_s \ar[d] \\
 & \Hom_{\Lambda_s}( (\Gamma_1(Np^r)^{\ab} \otimes \Z_p)^{\ord}/\mfa_s,\Lambda_s) \ar[d]^{\wr}\\
   \Hom_{\Z_p}( (\Gamma_1(Np^s)^{\ab} \otimes \Z_p)^{\ord},\Z_p) \ar[r]^{\sim} & \Hom_{\Lambda_s}( 
   (\Gamma_1(Np^s)^{\ab} \otimes \Z_p)^{\ord},\Lambda_s)
}
\end{equation*}
in which the two horizontal isomorphisms are those provided by Lemma \ref{Equality2}, because $\Lambda_r = \Z_p[\Gamma/\Gamma_r]$. The first vertical map $V^*$ is the dual morphism of $V$ and
 the two vertical isomorphisms are a part of Theorem \ref{Hida} and its proof.

Now to prove the Lemma, it suffices to prove that 
\begin{equation}
\label{final-lem-eqn}
\Hom_{\Z_p}((\Gamma_1(Np^r)^{\ab} \otimes \Z_p)^{\ord},\Z_p) \stackrel{V^*}{\longrightarrow}
\Hom_{\Z_p}((\Phi^{s \> \ab}_r \otimes \Z_p)^{\ord},\Z_p) 
\end{equation}
is surjective and  $\mathrm{kernel}(V^*) = \mfa_s\Hom_{\Z_p}( (\Gamma_1(Np^r)^{\ab} \otimes \Z_p)^{\ord},\Z_p))$.

Since $V$ commutes with $U$ and taking ordinary parts commutes with taking $\Z_p$-duals,
the morphism in~\eqref{final-lem-eqn} is the ordinary part of  the morphism
\begin{equation}
\label{final-lem-eq-2}
\Hom_{\Z_p}(\Gamma_1(Np^r)^{\ab} \otimes \Z_p,\Z_p) \stackrel{V^*}{\longrightarrow}
\Hom_{\Z_p}(\Phi^{s \> \ab}_r \otimes \Z_p,\Z_p). 
\end{equation}

Now, it suffices to  show that the morphism $V^*$ in~\eqref{final-lem-eq-2}
is surjective with kernel equal to $\mfa_s\Hom_{\Z_p}( (\Gamma_1(Np^r)^{\ab} \otimes \Z_p),\Z_p)$,
since taking ordinary parts is also exact and commutes with the action of $\Gamma$.  
But, this claim was proved in~\cite[\S 8]{Eme99} for torsion-free groups $H$ and $G$ such that 
$H \subseteq G$, instead of $\Gamma_1(Np^r) \subseteq \Phi_r^s$. Observe that, when $p = 2$ and $r \geq s \geq 2$,
the groups  $\Gamma_1(Nq)$ and $\Phi_r^s$ are torsion-free, since  $\Gamma_1(M)$ is torsion free for all  $M \geq 3$.  
\end{proof}

We now have all the information needed to prove Theorem~\ref{mainfree}.
Consider the chain of $\Lambda$-modules
\begin{equation*}
  \dotsm \lra \Hom_{\Lambda_r}((\Phi^{r \> \ab}_r\otimes \Z_p)^{\ord},\Lambda_r) 
  \lra  \dotsm  \lra \Hom_{\Z_p}((\Gamma_1(Nq)^{\ab}\otimes \Z_p)^{\ord},\Z_p). 
\end{equation*}

\begin{lem}
\label{cano-2}
There is a canonical isomorphism
 \begin{equation*}
  \Hom_{\Lambda}(\mfW^{\emph{\ord}}, \Lambda) = \varprojlim_r 
  \Hom_{\Lambda_r}((\Gamma_1(Np^r)^{\emph{\ab}} \otimes \Z_p)^{\emph{\ord}},\Lambda_r). 
\end{equation*}
\end{lem}

\begin{proof}
We have the following canonical isomorphisms
\begin{equation*}
  \Hom_{\Lambda}(\mfW^{\ord}, \Lambda)  = \varprojlim _r \Hom_{\Lambda_r}(\mfW^{\ord}/\mfa_r, \Lambda_r) 
   = \varprojlim_r \Hom_{\Lambda_r}((\Gamma_1(Np^r)^{\ab}\otimes \Z_p)^{\ord},\Lambda_r),
\end{equation*}
where the last isomorphism follows from the Theorem~\ref{mainthmone}. 
\end{proof}

\begin{prop}
\label{Equality1}
For $r > 1$, there is a canonical isomorphism
\begin{equation*}
 \Hom_{\Lambda}(\mfW^{\emph{\ord}}, \Lambda)/\mfa_r = 
 \Hom_{\Lambda_r}((\Gamma_1(Np^r)^{\emph{\ab}}\otimes \Z_p)^{\emph{\ord}},\Lambda_r).
\end{equation*}
\end{prop}

\begin{proof}
The claim follows from Lemma~\ref{MainLemma}, Lemma~\ref{cano-2}, and  Lemma~\ref{Keylemma}.
\end{proof}

\begin{thm}
\label{mainthmtwo}
The module $\mfW^{\ord}$ is $\Lambda$-free.
\end{thm}

\begin{proof}
Since any finitely generated reflexive $\Lambda$-module is free, it suffices to show that 
$\mfW^{\ord}$ is a reflexive $\Lambda$-module. By Proposition~\ref{Equality1} and Lemma~\ref{Equality2}, we have:
\begin{equation*}
\begin{split}
\Hom_{\Lambda}(\Hom_{\Lambda}(\mfW^{\ord}, \Lambda), \Lambda) &= \varprojlim_r \Hom_{\Lambda}(\Hom_{\Lambda_r}(\mfW^{\ord}, \Lambda)/\mfa_{r}, \Lambda_r)\\
  &= \varprojlim_r\Hom_{\Lambda_r}(\Hom_{\Lambda_r}((\Gamma_1(Np^r)^{\ab}\otimes \Z_p)
    ^{\ord}, \Lambda_r),\Lambda_r)\\
  &= \varprojlim_r\  (\Gamma_1(Np^r)^{\ab}\otimes \Z_p)^{\ord} = \mfW^{\ord}.
\end{split}
\end{equation*}
\end{proof}
\section{Acknowledgements}
The author thanks Prof. M. Emerton for his encouragement to work out the details in~\cite{Eme99} for the prime $p=2$. These results turned out be quite useful in~\cite{GK12}.
\bibliographystyle{plain, abbrv}

%
%
%
%
%

\end{document}